\documentclass{article}
\usepackage{amsmath}
\usepackage{amssymb,comment,amscd}
\usepackage{mathrsfs,wasysym}

\usepackage{euscript,color}

\usepackage{libertine}
\usepackage[T1]{fontenc}

\usepackage[utf8]{inputenc}

\usepackage{amsmath}
\usepackage{amssymb}
\usepackage{amsthm}
\usepackage{mathrsfs}

\usepackage{graphicx}
\graphicspath{ {/desktop/graphs/} }

\usepackage{version}
\usepackage[margin=3cm]{geometry}
\usepackage{color}
\usepackage{mathtools}
\usepackage{tikz}
\usetikzlibrary{arrows}
\usetikzlibrary{trees}
\usepackage{float}
\usepackage{textcomp}
\usepackage{bbm}

\numberwithin{equation}{section}

\newtheorem{proposition}{Proposition}[section]
\newtheorem{claim*}{Claim}[section]

\newtheorem{lemma}{Lemma}[section]
\newtheorem{theorem}{Theorem}[section]
\newtheorem{corollary}{Corollary}[section]

\newtheorem{remark}{Remark}
\newcommand\numberthis{\addtocounter{equation}{1}\tag{\theequation}}

\DeclareMathOperator{\E}{\mathbb{E}}

\DeclareMathOperator{\Q}{\mathcal{Q}}
\DeclareMathOperator{\Pb}{\mathbb{Pb}}

\DeclareMathOperator{\VAR}{Var}
\DeclareMathOperator{\COV}{Cov}

\DeclarePairedDelimiter{\bigfloor}{\Big\lfloor}{\Big\rfloor}

\newcommand{\vb}{\vspace{3mm}}

\begin{document}

\title{Occupation times of alternating renewal processes \\ with L\'evy applications}
\author{N. J. Starreveld, R. Bekker and M. Mandjes}
\maketitle

\begin{quotation}{\small

\noindent
{\bf Abstract}\:\: 
This paper presents a set of results relating to the occupation time $\alpha(t)$ of a process $X(\cdot)$. The first set of results concerns exact characterizations of $\alpha(t)$, e.g., in terms of its transform up to an exponentially distributed epoch. In addition we establish a central limit theorem (entailing that a centered and normalized version of $\alpha(t)/t$ converges to a zero-mean Normal random variable as $t\to\infty$) and the tail asymptotics of ${\mathbb P}(\alpha(t)/t\ge q)$. We apply our findings to spectrally positive L\'evy processes reflected at the infimum and establish various new occupation time results for the corresponding model.

\vb 


\vb

\noindent
{\bf Affiliations}\:\: 
{N. Starreveld} is with Korteweg-de Vries Institute for Mathematics, Science Park 904,
1098 XH Amsterdam, University of Amsterdam, the Netherlands. Email: {N.J.Starreveld@uva.nl.}

\noindent
{R. Bekker} is with Department of Mathematics, Vrije Universiteit Amsterdam, De Boelelaan 1081a, 1081
HV Amsterdam, the Netherlands. Email: {r.bekker@vu.nl.}

\noindent
 {M. Mandjes} is with Korteweg-de Vries Institute for Mathematics, University of Amsterdam, Science
Park 904, 1098 XH Amsterdam, the Netherlands. He is also affiliated with E{\sc urandom}, Eindhoven
University of Technology, Eindhoven, the Netherlands, and CWI, Amsterdam the Netherlands.
Email: {m.r.h.mandjes@uva.nl.}

}
\end{quotation}

\section{Introduction}
In this paper we consider a stochastic process $X(\cdot)\equiv\{X(t): t\geq0\}$ taking values on the state space $E$, and a partition of the state space $E= A\cup B$ into two (disjoint) sets $A,B$. Specifically, $X(\cdot)$ is an alternating renewal process where the sojourn times in set $B$ depend on the sojourn times in $A$. 
The object of study is the {\it occupation time}, denoted by $\alpha(t)$, of the set $A$ up to time $t$, defined by
\begin{equation}
\label{occupation time definition}
\alpha(t) = \int_0^t 1_{\{X(s)\in A\}}{\rm d}s;
\end{equation}  
as the set $A$ is held fixed, we suppress it in our notation.
Such occupation measures appear naturally when studying stochastic processes, and are useful in the context of a wide variety of applications. Our primary source of motivation stems from the study of occupation times of (reflected) spectrally positive L\'evy processes, where such an alternating renewal structure appears naturally.

\paragraph{Scope \& contributions} Our results essentially cover two regimes. In the first place we present results characterizing the transient behavior of $\alpha(t)$, in terms of expressions for the transform of $\alpha(e_q)$, with $e_q$ being exponentially distributed with mean $q^{-1}$.  Secondly, the probabilistic properties of $\alpha(t)$ for $t$ large are captured by a central limit theorem and large deviations asymptotics. We also include a series of new results in which we specialize to the situation that $X(\cdot)$ corresponds to a spectrally positive L\'evy process reflected at its infimum; for instance, we determine an explicit expression of the double transform of the occupation time $\alpha(t)$. For the case of an unreflected process, we recover a distributional relation between the occupation time of the negative half line up to an exponentially distributed amount of time, and the epoch at which the supremum is attained (over the same time interval). \textcolor{black}{Our results have been recently used to study processes with two reflecting barriers \cite{Sta}.}

\paragraph{Relation to existing literature}
The occupation time of a stochastic process was first considered  in \cite{Tak}, resulting in an expression for the distribution function of $\alpha(t)$ that enabled the derivation of a central limit theorem; a similar result was also established in \cite{Zac} using renewal theory. In \textcolor{black}{\cite{Cai, Lan, Loe}} occupation times of spectrally negative L\'evy processes were studied, while in \cite{Kyp Ref} refracted L\'evy processes were dealt with; these results are typically occupation times until a first passage time. \textcolor{black}{Occupation times up to a fixed time horizon $t$ have been studied in \cite{Gue} for spectrally negative L\'evy processes and in \cite{Wu} for a general L\'evy process which is not a compound Poisson process.} The cases in which $X(\cdot)$ is a Brownian motion, or Markov-modulated Brownian motion have been extensively studied; see e.g.\ \cite{Bil, Bre,Das, Pec} and references therein. A variety of results specifically applying to Brownian motion and reflected standard normal Brownian motion can be found in \cite{Salm}. Occupation times of dam processes were considered in \cite{Cohen}. Occupation times and its application to service levels over a finite interval in multi-server queues were dealt with in \cite{Baron, Roubos}. Applications in machine maintenance and telegraph processes can be found in \cite{Tak, Zac}.

This paper generalizes the results established in \cite{Cohen, Tak, Zac}. The setup of \cite{Tak, Zac}  assumes that the successive time intervals that $X(\cdot)$ spends in the two sets $A$ and $B$ form two independent sequences of i.i.d.\ random variables, whereas we relax this assumption; the motivation for pursuing  this extension lies in the fact that in many models  we observe dependency between such intervals.   
Our methodology also generalizes the methods of \cite{Cohen}: there, relying on renewal theory, specific mean quantities are found, whereas we uniquely characterize the corresponding full distributions. \textcolor{black}{ The results we prove can be applied to a broad class of processes including spectrally one-sided L\'evy processes with or without reflecting barriers.} \textcolor{black}{The large deviations result we prove is an additional novelty of this paper, to the authors knowledge large deviations asymptotics for occupation times have not been studied before in the literature.}

\paragraph{Organization} 
The structure of the paper is as follows. Section \ref{sec: model}  describes the model and presents the application that motivated our research, i.e., storage models and spectrally positive L\'evy processes reflected at its infimum. In Section \ref{sec: Overview of analysis and results} we present our main results and show how our findings can be applied to reflected spectrally positive L\'evy processes. Then, in Section \ref{proof of LST}, we derive the expression for the double transform of the occupation time. Section \ref{Proof of CLT} provides the proof of the central limit theorem, and Section \ref{Proof of LDP} of the large deviations result. The more technical proofs are included in an appendix.

\section{Model and Applications}
\label{sec: model}
First we provide a general model description in Section \ref{sec:Model description}. Then we introduce the more specific class of reflected spectrally positive L\'evy processes in Section \ref{sec: Applications 1}, where we distinguish two cases, viz.\ storage models and general reflected spectrally positive L\'evy processes. 

\subsection{Model description}
\label{sec:Model description}
We consider a stochastic process $X(\cdot)\equiv \{X(t): t\geq0\}$ taking values on the state space $E$, and a partition of $E$ into two disjoint subsets, denoted  by $A$ and $B$, i.e., $E=A\cup B$ and $A\cap B=\emptyset$. Then, $X(\cdot)$ alternates between $A$ and $B$. 
The successive sojourn times in $A$ are $(D_i)_{i\in{\mathbb N}}$, and those in $B$ are $(U_i)_{i\in{\mathbb N}}$.  If 
$(D_i)_{i\in{\mathbb N}}$ and $(U_i)_{i\in{\mathbb N}}$ are independent sequences of i.i.d.\ random variables, the resulting process is  an {\em alternating renewal process}, and has been considered in e.g.\ \cite{Tak, Zac}. In this paper, however, we consider the more general situation in which $(D_i, U_i)_{i\in{\mathbb N}}$ is a sequence of i.i.d.\ bivariate random vectors  that are distributed according to the generic random vectors $(D,U)$, but without requiring $D$ and $U$ to be independent. In the paper we prove our results for the case $X(0)\in A$, but the case $X(0)\in B$ can be treated along the same lines; also the case that $D_1$ has a different distribution can be dealt with, albeit at the expense of more complicated expressions.

\subsection{Spectrally positive L\'evy processes}
\label{sec: Applications 1}

A stochastic process $X(\cdot)$ defined on a probability space $(\Omega,\mathcal{F},\Pb)$ is called a L\'evy process if $X(0)=0$, it has almost surely c\`adl\`ag paths, and it has stationary and independent increments. Typical examples of L\'evy processes are Brownian motion and the (compound) Poisson process.  The L\'evy-Khinchine representation relates L\'evy processes with {\em infinitely divisible distributions} and it provides the following representation for the characteristic exponent $\Psi(\theta):=-\log\E(e^{{\rm i}\theta X(1)})$:
\[\Psi(\theta) = {\rm i}d\theta+\frac{1}{2}\sigma^2\theta^2+\int_{\mathbb{R}}(1-e^{{\rm i}\theta x}+{\rm i}\theta x1_{\{|x|<1\}})\Pi({\rm d}x),\] 
where $d\in \mathbb{R}$, $\sigma^2>0$ and $\Pi(\cdot)$ is a measure concentrated on $\mathbb{R}\backslash\{0\}$ satisfying $\int_{\mathbb{R}}(1\wedge x^2)\Pi({\rm d}x)<\infty$. We refer to \cite{Ber,Kyp} for an overview of the theory of L\'evy processes. When the measure $\Pi(\cdot)$ is concentrated on the positive real line then $X(\cdot)$ exhibits jumps only in the upward direction and we talk about a {\em spectrally positive} L\'evy process.  For a spectrally positive L\'evy process $X(\cdot)$ with a negative drift, i.e. $\E X(1)<0$, the Laplace exponent $\phi(\alpha):=\log \E e^{-\alpha X(1)}$
is a well defined, finite, increasing and convex function for all $\alpha\geq0$, so that the inverse function $\psi(\cdot)$ is also well defined; if it does not have a negative drift, we have to work with the right-inverse.

Given a L\'evy process $X(\cdot)$ we define the process $Q(\cdot)$, commonly referred to
as $X(\cdot)$ {\it reflected at its infimum} \cite{Ber,Kyp}, by 
$
Q(t) = X(t) + L(t),$ where $L(t)$ is the regulator process (or {\em local time at the infimum}) which ensures that $Q(t)\geq0$ for all $t\geq0$. Hence the process $L(\cdot)$ can increase at time $t$  only when $Q(t)=0$, that is,
$\int_0^T Q(t) {\rm d}L(t) = 0$ for all $T>0.$
 This leads to a Skorokhod problem with the following solution: with $Q(0)=w$,
\[L(t) = \max\left\{w,\sup_{0\leq s\leq t}-X(s)\right\},\:\:\:\:Q(t) = X(t) +\max\left\{w, \sup_{0\leq s\leq t}-X(s)\right\}.\]
For the process $Q(t)$ and a given level $\tau\geq0$, the occupation time of the set $[0,\tau]$
 is defined by
\begin{equation}
\label{occupation time storage model}
\alpha(t) = \int_0^t 1_{\{Q(s)\leq \tau\}}{\rm d}s.
\end{equation}

\paragraph{Storage models} Storage models are used to model a reservoir that is facing supply (input) and demand (output). Supply and demand are either described by sequences of random variables $(\eta_i)_{i\in{\mathbb N}}$ and $(\xi_i)_{i\in{\mathbb N}}$, or by an input process $A(\cdot)$ and an output process $B(\cdot)$. Storage models can be used to control the level of stored material by regulating  supply and demand. Early applications \cite{Mor} of such models concern finite dams which are constructed for storage of water; see \cite{Gani1} for additional applications of storage models and
\cite{Prabhu1963, King,Mor} for a historic account of the exact mathematical formulation of the content process. We refer to \cite[Ch. IV]{Pra} and \cite[Ch. IV]{Kyp} for an overview of dams and general storage models in continuous time. 

A general storage model consists of a (cumulative) {\em arrival process} $A(\cdot)$  and a (cumulative) {\em output process} $B(\cdot)$, leading to the {\em net input} process $V(\cdot)$ defined by $V(t) = A(t) - B(t)$; the work stored in the system at time $t$, denoted by $Q(t)$, is defined by applying the reflection mapping, as defined above, to $V(\cdot)$.
First we look into the case that $B(\cdot)$ corresponds to a positive linear trend, whereas $A(\cdot)$ is a pure jump subordinator.  Thus, $V(t) = X(t)+ w$ where $X(t) = A(t) - rt$ and $w= V(0)$ is the initial amount of work stored in the system. By construction, $X(\cdot)$ is a spectrally positive L\'evy process; without loss of generality, we assume that $r=1$. 

The reflected  process $Q(\cdot)$  has an interesting path structure: it has a.s.\ paths of bounded variation and has only jumps in the upward direction, such that upcrossings of a level occur with a jump whereas downcrossings of a level occur with equality. Moreover, from the L\'evy-Khinchine representation we have that if the L\'evy measure $\Pi(\cdot)$ satisfies $\Pi((0,\infty))=\infty$, then $X(\cdot)$ exhibits countably infinite jumps in every finite interval of time.
For the analysis of the occupation time $\alpha(t)$ we observe that the process $Q(\cdot)$ alternates between the two sets $A=[0,\tau]$ and $B=(\tau,\infty)$. We define the following first passage times, for $\tau\geq0$,
\begin{equation}
\tau_{\alpha} = \inf_{t>0}\{t: Q(t) > \tau\,|\, Q(0) = \tau\},
\:\:\:
\label{upcrossing}
\tau_{\beta} = \inf_{t>\tau_{\alpha}}\{t:Q(t)= \tau\,|\, Q(0) = \tau\}.
\end{equation}   
Observe that $Q(\tau_{\alpha})>\tau$.
In case $\E A(1) <r$ the process keeps on having downcrossings of level $\tau$. Call the sequence of successive downcrossings  $(T_i)_{i\in{\mathbb N}}$. 
As shown in \cite[Thms.\ 1 and 2]{Cohen} relying on the bounded variation property of the paths, $(T_i)_{i\in{\mathbb N}}$ is a renewal process, and hence $(D_i)_{i\in{\mathbb N}}$ and $(U_i)_{i\in{\mathbb N}}$ are sequences of well defined random variables. In addition, $(D_{i+1},U_{i+1})$  is independent of $(D_1,U_1,D_2,U_2,\ldots,D_i,U_i)$; at the same time the {\em overshoot} (over level $\tau$) makes $D_i$ and $U_i$ dependent. In Fig.\ \ref{Figure Path} an illustrative realization of $Q(\cdot)$ is depicted for the case of finitely many jumps in a bounded time interval.

\begin{figure}[H]
\includegraphics[scale=0.75]{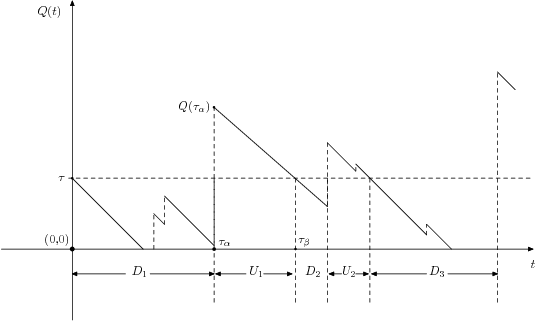}
\caption{Sample path for the case of finitely many jumps in a bounded interval of time.}
\label{Figure Path}
\end{figure}

\paragraph{General spectrally positive L\'evy processes} When the process $X(\cdot)$ is an arbitrary spectrally positive L\'evy process (i.e., we deviate from the setting of $A(\cdot)$ being  a pure jump process and $B(\cdot)$ a positive linear trend),  then $Q(\cdot)$ may have a more complicated path structure. We may have paths of unbounded variation, and as a result the intervals $D_i$ and $U_i$ are not necessarily well defined. But even in this case it is still possible, as will be shown in Section \ref{sec: Applications}, to study the occupation time of $[0,\tau]$ for the reflected process $Q(\cdot)$, and also the occupation time of $(-\infty, 0)$ for the free process $X(\cdot)$. This is done relying on an approximation procedure: we first approximate a general spectrally positive L\'evy process by a process with paths of bounded variation, and then use a continuity argument to show that the results for the bounded variation case carry over to the general spectrally positive case. 

In line with Section~\ref{sec:Model description}, we prove the results for the case the process starts at $\tau$. Having a different initial position requires a different distribution of $(D_1, U_1)$. This case can be treated along the same lines leading to more extensive notation and expressions and is therefore omitted.

\section{Overview of results}
\label{sec: Overview of analysis and results}

In this section we present the main results of the paper. 
As pointed out in the introduction, they cover two regimes. 
\begin{itemize}
\item[$\circ$]
In the first place we present results characterizing the transient behavior of $\alpha(t)$. We identify an expression for  the distribution function of $\alpha(t)$; the proof resembles that was developed in  \cite[Thm.~1]{Tak} for the alternating renewal case. The expressions for the transform of $\alpha(e_q)$, with $e_q$ exponentially distributed with mean $q^{-1}$, are found by arguments similar to those used in \cite{blom}, but taking into account the dependence between $D$ and $U$. The resulting double transform is a new result, which also covers the alternating renewal case with  independent $D$ and $U$, as was considered in \cite{Tak}, and is in terms of the joint transform of $D$ and $U$. In Section~\ref{sec: Applications} we specify this joint transform for the models of Section~\ref{sec: Applications 1}. 

\item[$\circ$] For the asymptotic behavior of $\alpha(t)$ we prove a central limit theorem and large deviations asymptotic. The central limit result generalizes \cite[Theorem 2]{Tak},
that covers the alternating renewal case. The result features the quantity $c={\rm Cov}(D,U)$; in Section \ref{sec: Applications} we identify closed-form expressions for $c$ for the models described in Section \ref{sec: Applications 1}. The large deviations principle is proven by using the G\"artner-Ellis theorem  \cite{dem}. 
\end{itemize}

\subsection{Distribution of $\alpha(t)$}
\label{subsection distribution}

We assume $(D_i,U_i)_{i\in{\mathbb N}}$ to be sequences of i.i.d.\ random vectors, distributed as the generic random vector $(D, U)$; $D$ and $U$ have distribution functions $F(\cdot)$ and $G(\cdot)$, respectively. We define $X_n = D_1+\hdots+D_n$ and $Y_n = U_1+\hdots+U_n$. In addition to $\alpha(t)$, we define the occupation time of $B$ by
\begin{equation}
\label{occupation time}
\beta(t) = \int_0^t 1_{\{X(s)\in B\}}{\rm d}s = t-\alpha(t).
\end{equation}
The proof of the following result is analogous to that of \cite[Thm.\ 1]{Tak}.

\begin{proposition}
\label{proposition distribution <tau}
For $X(0) \in A$ and $0\leq x<t$,
\begin{eqnarray}
\label{distribution b(t)}
\Pb(\alpha(t)<x) &=& F(x)-\sum_{n=1}^\infty\big[\Pb(Y_n\leq t-x,X_{n}<x)-\Pb(Y_n\leq t-x, X_{n+1}<x)\big],
\\
\label{distribution a(t)}
\Pb(\beta(t)\leq x) &=& \sum_{n=0}^\infty\big[\Pb(Y_n\leq x,X_{n}<t-x)-\Pb(Y_n\leq x, X_{n+1}<t-x)\big].
\end{eqnarray}
\end{proposition}

\begin{remark}{\em  If $X(0)\in A$, then $\beta(t)$ has an atom at 0; the mass at 0 is given by $\Pb(\beta(t) = 0) = 1-F(t)$. Similarly, if $X(0)\in B$, then $\alpha(t)$ has an atom at 0; the mass at 0 equals $\Pb(\alpha(t)=0) = 1-G(t)$. }
\end{remark}
\begin{remark}{\em 
If $D_i$ and $U_i$ are independent random variables, then (\ref{distribution a(t)}) takes the simpler form
\[\Pb(\beta(t)\leq x) = \sum_{n=0}^\infty G^{(n)}(x)\big[ F^{(n)}(t-x)- F^{(n+1)}(t-x)\big],\]
where $F^{(n)}(\cdot)$ and $G^{(n)}(\cdot)$ are the $n$-fold convolutions of the distribution functions $F(\cdot)$ and $G(\cdot)$ with itself, respectively. }
\end{remark}

\subsection{Transform of $\alpha(t)$}
\label{subsection transform}

From Proposition \ref{proposition distribution <tau} we observe that the distribution function of $\alpha(t)$ is rather complicated to work with. In this section we therefore focus on
\begin{equation}
\label{transform of alpha(t)}
\int_0^\infty e^{-qt} \E e^{-\theta \alpha(t)}{\rm d}t=\frac{1}{q}\E e^{-\theta \alpha(e_q)},
\end{equation}
where $e_q$ is an exponentially distributed random variable with mean $q^{-1}$; using a numerical inversion algorithm \cite{Aba, Aba 2, Aba 3, Den}, one can then obtain an accurate approximation for the distribution function of $\alpha(t)$. The method we use to determine the transform in (\ref{transform of alpha(t)}) relies on the following line of reasoning: (i)~we start an exponential clock at time 0, (ii)~knowing that $X(0)\in A$,  we distinguish between  $X(\cdot)$  still being in $A$ at $e_q$, or   $X(\cdot)$ having left $A$ at $e_q$, (iii) in the latter case we distinguish between $X(\cdot)$ being still in $B$
at $e_q$, or  $X(\cdot)$ having left $B$. Using the memoryless property of the exponential distribution and the independence of $D_{i+1}$ and $U_i$ we can then sample the exponential distribution again, and this procedure continues until the exponential clock expires. Let 
\begin{equation}
\label{transforms 12}
L_1(\theta) = \E e^{-\theta D} \hspace{2mm} \text{and} \hspace{2mm} L_{1,2}(\theta_1,\theta_2) = \E e^{-\theta_1 D-\theta_2 U}.
\end{equation}
\textcolor{black}{The main result concerning the transform of the occupation time $\alpha(t)$ is given below.}
\begin{theorem}
\label{theorem on LST}
For the transform of the occupation time $\alpha(t)$ we have
\begin{equation}
\label{transform expression}
\int_0^\infty e^{-q t} \E e^{-\theta \alpha(t)}{\rm d}t = \frac{1}{1-L_{1,2}(q+\theta,q)}\left[\frac{1-L_1(q+\theta)}{q+\theta}+\frac{L_1(q+\theta)-L_{1,2}(q+\theta,q)}{q}\right].
\end{equation}
\end{theorem}

\begin{remark}{\em 
\label{remark LST}
Theorem \ref{theorem on LST} shows that the transform of the random variable $D$ and the joint transform of $D,U$ are needed in order to compute the transform of $\alpha(t)$. These quantities are model specific, i.e.,  we have to assume $X(\cdot)$ has some specific structure to be able to compute them. In Section  \ref{sec: Applications} we come back to this issue, and we compute these quantities for the models considered in Section \ref{sec: Applications 1}.}
\end{remark}
\paragraph{Availability function} Another performance measure that is used in many applications, particularly in the theory of machine maintenance scheduling and reliability, see e.g. \cite{Zac}, is the availability function $\Pb(X(\cdot) \in A\,|\,X(0)\in A)$, of which the transform is given in the following proposition.
\begin{proposition}
\label{corollary AF}
For $q\geq0$, with $L_1(\cdot)$ and $L_{1,2}(\cdot,\cdot)$  given in $(\ref{transforms 12})$,
\begin{eqnarray}
\label{transform AF}
\int_{0}^{\infty} e^{-qt} \Pb(X(t)\in A| X(0)\in A){\rm d}t &=&  \frac{1}{q}\frac{1-L_1(q)}{1-L_{1,2}(q,q)}
\\
\label{transform AF2}
\int_{0}^{\infty} e^{-qt} \Pb(X(t)\in B| X(0)\in A){\rm d}t &=&\frac{1}{q}\frac{L_1(q) - L_{1,2}(q,q)}{1-L_{1,2}(q,q)}.
\end{eqnarray}
\end{proposition}
\subsection{Central Limit Theorem}
\label{subsection CLT}

Where Proposition \ref{proposition distribution <tau} and Theorem \ref{theorem on LST} entail that the full joint distribution of $D$ and $U$ is essential when characterizing $\alpha(t)$, the next theorem shows that in the central limit regime only the covariance is needed. From \cite{Tak} we have that, for the case $D$ and $U$ are independent,  under appropriate scaling, $\alpha(t)$ and $\beta(t)$ are asymptotically normally distributed; our result shows that the asymptotic normality carries over to our model, with an adaptation of the parameters to account for the dependence between $D$ and $U$. We write 
$\alpha=\E D$, $\beta = \E U,$ $\sigma_{\alpha}^2 = \VAR D$, $ \sigma_{\beta}^2 =\VAR U$, and $c={\rm Cov}(D,U).$ We rule out the case of perfect correlation, i.e., we assume that $c<\sigma_\alpha \sigma_\beta$.

\begin{theorem}
\label{CLT theorem}  
Assuming $\sigma_{\alpha}^2 +\sigma_{\beta}^2<\infty$, with $\Phi(\cdot)$ the standard Normal distribution function,
\begin{equation}
\label{CLT 2}
\lim_{t\rightarrow \infty} \Pb\left( \frac{\alpha(t)-\frac{\alpha t}{\alpha+\beta}}{\sqrt{\frac{\beta^2\sigma_{\alpha}^2+\alpha^2\sigma_{\beta}^2-2\alpha\beta c}{(\alpha+\beta)^3}t}}< x\right)=\Phi(x)
\hspace{2mm} \text{and} \hspace{2mm}
\lim_{t\rightarrow \infty} \Pb\left( \frac{\beta(t)-\frac{\beta t}{\alpha+\beta}}{\sqrt{\frac{\beta^2\sigma_{\alpha}^2+\alpha^2\sigma_{\beta}^2-2\alpha\beta c}{(\alpha+\beta)^3}t}}\leq x\right)=\Phi(x) \end{equation}
\end{theorem}
\begin{remark}{\em 
In Section \ref{sec: Applications} we compute $\alpha$, $\beta$, $\sigma_{\alpha}^2, \sigma_{\beta}^2$, and $c$ for the models introduced in Section \ref{sec: Applications 1}.}
\end{remark}

\subsection{Large Deviations Result}
\label{subsection LDP}
Where Thm.\ \ref{CLT theorem} concerns the behaviour of $\alpha(t)/t$ and $\beta(t)/t$ around their respective means, we now focus on their tail behavior. To this end, we study $\Pb(\alpha(t)/t> q )$ for $t$ large, with $q> {\alpha}/({\alpha+\beta})$. Our result makes use of the following objects:
\[\lambda(\theta) = \lim_{t\rightarrow\infty}\frac{1}{t}\log\E e^{\theta \alpha(t)},\:\:\:\:
\lambda^*(q) = \sup_{\theta\in \mathbb{R}}\big( \theta q - \lambda(\theta)\big),\]
usually referred to as the cumulant generating function and its Legendre-Fenchel transform, respectively. The following result shows that the large deviations of $\alpha(t)/t$ require the joint transform of $D$ and $U$ being available. It requires that the moment generating function ${\mathbb E}\, e^{\theta D}$ is finite for $\theta$ in a neighborhood of the origin. 

\begin{theorem}
\label{two equations theorem}
Assuming $\alpha, \beta < \infty$, if the cumulant generating function 
$\lambda(\theta)$ exists as an extended real number, then,
\begin{equation}
\label{large deviation} 
\lim_{t\rightarrow\infty}\frac{1}{t}\log\Pb\left(\frac{\alpha(t)}{t}>q\right) = -\lambda^*(q),\end{equation}
for any $q> {\alpha}/({\alpha+\beta})$.
The cumulant generating function $\lambda(\theta)$
equals $\theta d(\theta)$, where for a given $\theta$, $d(\theta)$ solves
\begin{equation}
\label{c(theta)}
\E e^{\theta (1-d(\theta))\,D\,-\,\theta \,d(\theta)\,U}=1.
\end{equation}
\end{theorem}

\subsection{Spectrally positive L\'evy processes}
\label{sec: Applications}

In this section we compute the quantities needed in Thms.\ \ref{theorem on LST}, \ref{CLT theorem} and \ref{two equations theorem} for the models introduced in Section~\ref{sec: Applications 1}., with a focus on spectrally-positive L\'evy processes with negative drift. 
Before presenting the main analysis we recall some results on scale functions that we need in our analysis \cite{Kyp}. 

 \paragraph{Scale functions} When studying spectrally one-sided L\'evy processes results are often expressed in terms of so-called scale functions,  denoted by $W^{(q)}(\cdot)$ and $Z^{(q)}(\cdot)$, for $q\geq0$. \textcolor{black}{We will use the notation $W(\cdot), Z(\cdot)$ for the scale functions $W^{(0)}(\cdot)$ and $Z^{(0)}(\cdot)$}. Given a spectrally positive L\'evy process $X(\cdot)$ with Laplace exponent $\phi(\cdot)$, there exists an increasing and continuous function $W^{(q)}(\cdot)$ whose Laplace transform satisfies, for $x\geq0$, the equation
\[\int_0^\infty e^{-\theta x}W^{(q)}(x){\rm d}x = \frac{1}{\phi(\theta)-q} \hspace{5mm} \text{for } \theta>\psi(q),\]
with $\psi(\cdot)$ the inverse function of $\phi(\cdot)$. For $x<0$, we define $W^{(q)}(x)=0$.  In addition,
\[Z^{(q)}(x) = 1+q\int_0^x W^{(q)}(y){\rm d}y.\]
In most of the literature, the theory of scale functions is developed for spectrally negative L\'evy processes but similar results can be proven for spectrally positive L\'evy processes. 
We also define, for $\theta$ such that $\phi(\theta)<\infty$ and $q\geq\phi(\theta)$, the functions 
\begin{equation}
\label{exponential twist W}
W_{\theta}^{(q-\phi(\theta))}(x) = e^{-\theta x}W^{(q)}(x) \hspace{5mm} \text{and} \hspace{5mm} Z_{\theta}^{(q-\phi(\theta))}(x)=1+(q-\phi(\theta))\int_0^x W_{\theta}^{(q-\phi(\theta))}(y){\rm d}y.
\end{equation}
To formally define the functions $W_{\theta}^{(q-\phi(\theta))}(\cdot)$, $Z_{\theta}^{(q-\phi(\theta))}(\cdot)$ we need to perform an exponential change of measure; we refer to \cite{Avr, Kyp}  for further details. 
For the case $X(\cdot)$ is a spectrally positive (resp. spectrally negative) L\'evy process the scale function $W(\cdot)$ directly relates to the distribution function of the running maximum (resp.\ running minimum) of $X(\cdot)$. For a more extensive account of the theory of scale functions and exit problems for L\'evy processes we refer to \cite{Avr, Ber, Hub, Kuz, Kyp}.

\paragraph{Storage models}
\hspace{2mm} To apply Thms.\ \ref{theorem on LST}, \ref{CLT theorem} and \ref{two equations theorem} we need the joint transform of $U$ and $D$; given this transform we can compute the transform of the occupation time using (\ref{transform expression}), find the covariance $c=\COV(D,U)$ and the variances $\sigma_{\alpha}^2$, $\sigma_{\beta}^2$ needed in Thm.\ \ref{CLT theorem}, and solve  (\ref{c(theta)}). 

First we find representations for $D$ and $U$. Observe that $D$ is  distributed as the time between a downcrossing of level $\tau$ and the next upcrossing of it: with $\tau_{\alpha}$  as defined in (\ref{upcrossing}),
\begin{equation}
\label{D time}
D= \tau_{\alpha} \hspace{5mm} \text{\textcolor{black}{a.s}}.
\end{equation}
The time between an upcrossing and the next downcrossing, i.e.,  $U$,  is essentially a first exit time starting from a random point which is sampled from the overshoot distribution. Defining $\sigma_{x}^{(+)} = \inf_{t>0}\{t: Q(t) \le \tau\,|\, Q(0) = x\}$, we have
\begin{equation}
\label{U time}
U= \tau_{\beta} - \tau_{\alpha} =\sigma_{Q(\tau_{\alpha})}^{(+)} \hspace{5mm} \text{\textcolor{black}{a.s}}.
\end{equation}
The two distributional equalities (\ref{D time}) and (\ref{U time}) yield an expression for the covariance $c=\COV(D,U)$:
\[\COV(D,U) = \E[DU] - \E D\,\E U= \E \left[\tau_{\alpha}\, \sigma_{Q(\tau_{\alpha})}^{(+)}\right]- \E D\,\E U.\]

\begin{proposition}
\label{transforms}
(i)~The means $\alpha = \E D$ and $\beta = \E U$ equal
\begin{equation}
\label{mean time}
\E D = \frac{W(\tau)}{W'(\tau)},\:\:\:\E U = \frac{1}{W'(\tau)}\Big(\psi'(0)-W(\tau)\Big),
\end{equation}
(ii) The joint transform of $D$ and $U$ equals, with $p=\theta_1-\theta_2$,
\begin{align*}
\E\hspace{-1pt}e^{-\theta_1 D-\theta_2 U}\hspace{-1pt} &=\hspace{-1pt}\hspace{-1pt}\frac{W_{\psi(\theta_2)}^{(p)\hspace{2pt}'}(\tau)+(\psi(\theta_2)-p)W_{\psi(\theta_2)}^{(p)}(\tau)-\psi(\theta_2)Z_{\psi(\theta_2)}^{(p)}(\tau)}{W_{\psi(\theta_2)}^{(p)\hspace{2pt}'}(\tau)+\psi(\theta_2)W_{\psi(\theta_2)}^{(p)}(\tau)}.\numberthis\label{transformUD}
\end{align*}
(iii) The variances $\sigma_{\alpha}^2= \VAR D$ and $\sigma_{\beta}^2= \VAR U$ and covariance $c=\COV(D,U)$ equal
\begin{align*}
\sigma_{\alpha}^2 &= -2\frac{(W\star W)(\tau)}{W'(\tau)}+\frac{W(\tau)\left(2(W\star W)'(\tau)-W(\tau)\right)}{W'(\tau)^2}, \numberthis\label{variance D}\\
\sigma_{\beta}^2 &= \frac{2\psi'(0)}{W'(\tau)}\left(\int_ 0^{\tau}W(y){\rm d}y+\frac{W(\tau)}{W'(\tau)}\right)\hspace{-2pt}-\frac{1}{W'(\tau)}\hspace{-2pt}\left(\hspace{-3pt}\psi''(0)+2\psi'(0)^2\tau+\frac{\psi'(0)}{W'(\tau)}+\frac{W(\tau)^2}{W'(\tau)}\right) , \numberthis\label{variance U}\\
c &= \big(\psi'(0)-W(\tau)\big)\frac{1}{(W'(\tau))^2}\left( (W\star W)'(\tau) - W(\tau)\right)
+\frac{(W\star W)(\tau)}{W'(\tau)}-\frac{\psi'(0)}{W'(\tau)}\int_0^{\tau} W(x){\rm d}x  \numberthis\label{Covariance}.
\end{align*}
\end{proposition}
\begin{remark}{\em 
The arguments used can be adapted to the case that $D_1$ has a different distribution than $D_2, D_3,\ldots$. The only difference is that the first cycle $D_1,U_1$ has to be treated separately.} 
\end{remark}

\paragraph{General spectrally positive L\'evy process} As mentioned before, when $X(\cdot)$ is a general spectrally positive L\'evy process, with possibly paths of unbounded variation, then the reasoning presented above does not apply. The following theorem, however, shows that our result for the occupation time $\alpha(t)$ carries over to this case as well.

\begin{theorem}
\label{Theorem on occupation time SP}
Consider a spectrally positive L\'evy process  reflected at its infimum. With $A=[0,\tau]$ and $q,\theta\geq0$,
\begin{equation}
\label{OTSM}
\int_0^\infty e^{-q t}\E e^{-\theta \alpha(t)}{\rm d}t = \frac{1}{q}\frac{\psi(q)Z_{\psi(q)}^{(\theta)}(\tau)}{\theta W_{\psi(q)}^{(\theta)}(\tau)+\psi(q)Z_{\psi(q)}^{(\theta)}(\tau)}.
\end{equation}
\end{theorem}
Letting the initial level $\tau\rightarrow \infty$, the effect of the reflection becomes negligible and the time the reflected process spends below the initial level $\tau$ converges to time the unreflected reflected process that started in the origin spends below 0.
Consequently, taking $\tau\to\infty$ in (\ref{OTSM}), we obtain the occupation time of the set $(-\infty,0]$ for the free process $X(\cdot)$ and recover a version of the Sparre Andersen's identity, see e.g. \cite[Lemma VI.15]{Ber}, \cite[Remark 4.1]{Lan}, or \cite{Iva}.

\begin{corollary}
\label{distributional equality}
Consider a spectrally positive L\'evy process; the occupation time of $(-\infty,0]$ has the Laplace transform
\[\int_0^\infty e^{-q t}\E e^{-\theta \alpha(t)}{\rm d}t = \frac{1}{q}\frac{\psi(q)}{\psi(\theta+q)}.\]
In addition, letting $e_q$ be an exponentially distributed random variable with mean $q^{-1}$,
\[\alpha(e_q) \stackrel {\rm d}{=} e_q - G_{e_q},\]
where $G_{e_q} = {\rm arg}\sup_{0\leq s \leq e_q}X(s)$ is the epoch at which the supremum is attained.
\end{corollary}
\begin{remark}
\textcolor{black}{{\em The first expression derived in Cor.\ \ref{distributional equality} when combined with \cite[Remark 3.2 (1)]{Gue} yields the following equality for the scale functions $Z_{\psi(q)}^{(\theta)}(\cdot)$ and $W^{(q+\theta)}(\cdot)$
\begin{equation}
\label{eq: scale function}
\int_0^{\infty}\left(\frac{\psi(q+\theta)-\psi(q)}{\theta}e^{\psi(q)x}Z_{\psi(q)}^{(\theta)}(x) - W^{(q+\theta)}(x)\right){\rm d}x=\frac{(\theta-q)\psi(q+\theta)+q\psi(q)}{\psi(q)}.
\end{equation}
}}
\end{remark}

\begin{remark}
{\em
The expressions derived in Prop.\ \ref{transforms} are closed-form expressions in terms of the $q$-scale functions $W^{(q)}(\cdot)$ and $Z^{(q)}(\cdot)$. For some classes the scale functions are explicitly available \cite{Hub, Nguyen}, but in general we only know them  through their Laplace transform. For instance, \cite{Aba, Rog, Sur} provide us with  computational techniques to numerically evaluate the scale functions for spectrally one-sided L\'evy processes.
}

\end{remark}

\paragraph{Reflected Brownian motion} Suppose $X(\cdot)$ is a Brownian motion with drift $\mu$ and variance $\sigma^2$. We have the following expressions for the scale functions:
\begin{equation}
\label{eq: scale Wqpsi}
W_{\psi(q)}^{(\theta)}(x) = \frac{2}{\sqrt{D(q+\theta)}} e^{-\frac{x}{\sigma^2} \sqrt{D(q)}}\text{sinh}\left(\frac{x}{\sigma^2}\sqrt{D(q+\theta)}\right),
\end{equation}
and 
\begin{align*}
Z_{\psi(q)}^{(\theta)}(x) = \frac{\theta\sigma^2}{\sqrt{D(q+\theta)}}&\Bigg(\frac{1}{\sqrt{D(q+\theta)}-\sqrt{D(q)}} e^{\frac{x}{\sigma^2} \left(\sqrt{D(q+\theta)}-\sqrt{D(q)}\right)} \\
&+ \frac{1}{\sqrt{D(q+\theta)}+\sqrt{D(q)}} e^{-\frac{x}{\sigma^2} \left(\sqrt{D(q+\theta)}+\sqrt{D(q)}\right)}\Bigg), \numberthis\label{eq: ZscaleBMpsi}
\end{align*}
where $D(z) = 2\sigma^2z+\mu^2$. An expression for the scale function $W^{(q)}(\cdot)$ can be found in \cite{Kuz}. The expressions in \eqref{eq: scale Wqpsi} and \eqref{eq: ZscaleBMpsi}  can be derived, after some standard calculus, from \eqref{exponential twist W}. Theorem \ref{Theorem on occupation time SP} yields an expression for the double transform of the occupation time of $[0,\tau]$, up to time $t$,  for a reflected Brownian motion, i.e., 
\begin{equation}
\label{OTRBM}
\int_0^\infty  e^{-qt}\E e^{- \theta \alpha(t)}{\rm d}t =  \frac{1}{q} \frac{\mathcal{L}_1(q,\theta,\tau)}{\mathcal{L}_1(q,\theta,\tau)+\mathcal{L}_2(q,\theta,\tau)} ,
\end{equation}
where 
\begin{equation}
\mathcal{L}_1(q,\theta,\tau):=\sigma^2\psi(q)\left(\sqrt{D(q+\theta)}~\text{cosh}\left(\frac{\sqrt{D(q+\theta)}}{\sigma^2}\tau\right)+\sqrt{D(q)}~\text{sinh}\left(\frac{\sqrt{D(q+\theta)}}{\sigma^2}\tau\right)\right)
\end{equation}
and 
\begin{equation}
\mathcal{L}_2(q,\theta,\tau):=2\sigma^2\theta~\text{sinh}\left(\frac{\sqrt{D(q+\theta)}}{\sigma^2}\tau\right).
\end{equation}
For $\mu=0$ and $\sigma^2=1$, (\ref{OTRBM}) agrees with \cite[Equation~(1.5.1), pp. 337]{Salm}.
 Letting $\tau\to\infty$ at the right hand side of (\ref{OTRBM}) and inverting the double transform we find the following expression for the density of the occupation time $\alpha(t)$ of the free process on the negative half line: for $s\in[0,t]$,
\begin{align*}
\Pb(\alpha(t)\in{\rm d}s) &= \left(\frac{1}{\pi}\frac{1}{\sqrt{s(t-s)}}e^{-\frac{\mu^2}{2\sigma^2}t}-\zeta(s,t)+\zeta(t-s,t)-2\frac{\mu^2}{\sigma^2}\Phi\left(\frac{\mu}{\sigma}\sqrt{t-s}\right)\Phi\left(-\frac{\mu}{\sigma}\sqrt{s}\right)\right){\rm d}s, \numberthis\label{density}
\end{align*}
with
\[\zeta(s,t):=\sqrt{\frac{2}{\pi}}\frac{\mu}{\sigma}\frac{1}{\sqrt{t-s}}\Phi\left(-\frac{\mu}{\sigma}\sqrt{s}\right)e^{-\frac{\mu^2}{2\sigma^2}(t-s)}.\]
This agrees with the expression established in \cite[Equation~(3)]{Pec}. 
\begin{remark}{\em 
For spectrally one-sided L\'evy processes with paths of unbounded variation we know that $W^{(q)}(0)=0$, see \cite[Lemma 3.1]{Kuz}. For $\tau=0$, (\ref{OTSM}) and (\ref{OTRBM}) give the double transform of the occupation time of $\{0\}$ which is equal to one, as expected due to the regularity of point 0. }
\end{remark}
\section{Proof of Theorem \ref{theorem on LST}}
\label{proof of LST}

\begin{proof}[\textbf{Proof of Theorem \ref{theorem on LST}}]
For the intuitive ideas behind the proof, see Section \ref{subsection transform}. We show that 
$
\E e^{-\theta \alpha(e_q)} = \mathcal{K}_1(\theta,q) + \mathcal{K}_2(\theta,q),
$
where 
\begin{equation}
\label{K1}
\mathcal{K}_1(\theta,q) = \frac{q}{q+\theta} \frac{1-L_1(q+\theta)}{1-L_{1,2}(q+\theta,q)},\:\:\: 
\mathcal{K}_2(\theta,q) = \frac{L_1(q+\theta)-L_{1,2}(q+\theta,q)}{1-L_{1,2}(q+\theta,q)}.
\end{equation}
Considering the three disjoint events $\{e_q< D\}$, $\{D\leq e_q< D+U\}$ and $\{e_q\geq D+U\}$,
evidently
\[
\E e^{-\theta \alpha(e_q)} =\E\left[ e^{-\theta \alpha(e_q)}1_{\{e_q<D\}}\right]+\E \left[ e^{-\theta \alpha(e_q)}1_{\{D\leq e_q<D+U\}}\right]+\E \left[ e^{-\theta \alpha(e_q)}1_{\{e_q\geq D+U\}}\right].\]
We work out the three terms above, which we call $I_1,I_2$ and $I_3$, separately. We obtain $I_1$ by conditioning on $e_q$ and $D$, with $F(\cdot)$ as before the distribution function of $D$,
\begin{equation*}
I_1= \int_{t=0}^\infty qe^{-qt}\int_{(t, \infty)} e^{-\theta t}\Pb(D\in{\rm d}s)\,{\rm d}t=\int_{t=0}^\infty q e^{-(q+\theta)t}(1-F(t)){\rm d}t=  \frac{q}{q+\theta}\left (1-L_1(q+\theta)\right). \numberthis \label{I1}
\end{equation*}
Similarly, for $I_2$ we obtain 
\begin{align*}
I_2& = \int_{t=0}^\infty qe^{-qt}\int_{(0,t)} e^{-\theta s}\Pb(D\in{\rm d}s)\int_{(t-s,\infty)} \Pb(U\in {\rm d}u\,|\,D = s){\rm d}t\\
&= \int_{t=0}^\infty qe^{-qt}\int_{(0,t)} e^{-\theta s}\Pb(D\in{\rm d}s) \Pb(U>t-s\,|\,D = s){\rm d}t\\
&=\int_{t=0}^\infty q e^{-qt} \int_{(0,t)} e^{-\theta s}\Pb(U>t-s, D\in {\rm d}s){\rm d}t\,=\,L_1(q+\theta)-L_{1,2}(q+\theta,q). \numberthis \label{I2}
\end{align*}
To identify $I_3$ we   use the regenerative nature of $X(\cdot)$, i.e., we use the fact  that $(D_{i+1}, U_{i+1})$ is independent of $(D_1,U_1,D_2, U_2,\ldots,D_i,U_i)$ in combination with the memoryless property of the exponential distribution. After $X(\cdot)$ leaves subset $B$ we sample the exponential clock again. This yields 
\begin{align*}
I_3&=  \left[\int_{t=0}^\infty q e^{-qt}\int_{(0,t)} e^{-\theta s}\Pb(D\in{\rm d}s)\int_{(0,t-s)} \Pb(U\in{\rm d}u\,|\,D = s)\,{\rm d}t\right]\cdot \E e^{-\theta \alpha(e_q)}\\
&=\left[\int_{t=0}^\infty q e^{-qt}\int_{(0,t)} e^{-\theta s}\Pb(U\leq t-s, D\in {\rm d}s)\,{\rm d}t\right]\cdot \E e^{-\theta \alpha(e_q)}\\
&=\left[\int_{(0,\infty)} e^{-(q+\theta)s}\int_{t=0}^\infty q e^{-qt}\Pb(U \le t, D\in {\rm d}s){\rm d}t\right]\cdot \E e^{-\theta \alpha(e_q)}=L_{1,2}(q+\theta,q)\cdot\E e^{-\theta \alpha(e_q)}. \numberthis\label{I3}
\end{align*}
Adding the expressions that we found for $I_1,I_2$ and $I_3$ we obtain $
\E e^{-\theta \alpha(e_q)} = \mathcal{K}_1(\theta,q) + \mathcal{K}_2(\theta,q),
$ as desired, with $\mathcal{K}_1(\theta,q), \mathcal{K}_2(\theta,q)$ defined in (\ref{K1}). 
\end{proof}

\section{Proof of Central Limit Theorem}
\label{Proof of CLT}
\begin{proof}[\textbf{Proof of Theorem \ref{CLT theorem}}]

Define $\alpha^\circ:=\alpha/(\alpha+\beta)$. Also,
\[T_i:= \sum_{k=1}^i(D_k+U_k),\]
and $N(t)$ the number of regeneration points until time $t\ge 0$:
\[N(t):= \max\{i\ge 0: T_i\le t\}.\]
Splitting $\alpha(t)$ into the contributions due to the regeneration cycles, we obtain
\begin{align*}\alpha(t)-\alpha^\circ t =&\:\sum_{k=1}^{N(t)} \Big((1-\alpha^\circ)D_i -\alpha^\circ U_i\Big)\,1_{\{N(t)>0\}}+\\
&\:\int_{T_{N(t)}}^t 
\Big(1_{\{X(s)\in A\}}- \alpha^\circ\Big){\rm d}s \,1_{\{N(t)>0\}}+ (\alpha(t)-\alpha^\circ t)\,1_{\{N(t)=0\}}.\end{align*}
Now divide the right-hand side  of the previous display by $\sqrt{t}$. 
The next step is to prove that the last two contributions can be safely ignored as $t$ grows large. We do so by showing that both term converge to $0$ in probability as $t\to\infty.$
\begin{itemize}
\item[$\circ$]
In the first place,
\[\Big|\, \frac{1}{\sqrt{t}}
\:\int_{T_{N(t)}}^t 
\Big(1_{\{X(s)\in A\}}- \alpha^\circ\Big){\rm d}s
\,\Big|\le \frac{2}{\sqrt{t}} \int_{T_{N(t)}}^t \,{\rm d}t \le 2\,\frac{t -T_{N(t)} }{\sqrt{t}} .\]
Now observe that, for any $\epsilon >0$, by the Markov inequality,
\begin{equation}
\label{NN}{\mathbb P}\left(\frac{t -T_{N(t)} }{\sqrt{t}} \ge \epsilon\right)\le \frac{{\mathbb E}(t -T_{N(t)})}{\sqrt{t\,\epsilon}} .\end{equation}
It is a standard result from renewal theory that the `undershoot' $t -T_{N(t)}$ converges in the sense that
\begin{align*}\lim_{t\to\infty} {\mathbb E}(t -T_{N(t)})=\: &\int_0^\infty \frac{1}{{\mathbb E}D+{\mathbb E}U}\int_x^\infty {\mathbb P}(D+U>y)\,{\rm d}y\,{\rm d}x\\
=\: &\int_0^\infty \frac{1}{{\mathbb E}D+{\mathbb E}U}\,y\, {\mathbb P}(D+U>y)\,{\rm d}y
;\end{align*}
see e.g. \cite[Section A1e]{AA}.
Now applying integration by parts, and recalling that (by Cauchy-Schwartz) $\sigma^2_\alpha<\infty$ and $\sigma^2_\beta<\infty$ imply that ${\mathbb V}{\rm ar}(D+U)<\infty$, we conclude that this expression is bounded from above by a positive constant.
As a consequence, for all positive $\epsilon$ the right-hand side of (\ref{NN}) vanishes as $t\to\infty$, and therefore the term under consideration converges to $0$ in probability as $t\to\infty.$
\item[$\circ$]
In addition, again applying the Markov inequality, noticing that $\{N(t) = 0\}=\{T_1>t\}$,
\begin{align*}
{\mathbb P}\left(\Big|\frac{\alpha(t)-\alpha^\circ t}{\sqrt{t}}1_{\{T_1>t\}}\Big|\ge \epsilon \right) \le&\:
{\mathbb P}\left(2\,1_{\{T_1>t\}}\ge \frac{\epsilon}{\sqrt{t}} \right) \le \frac{2\sqrt{t}}{\epsilon} {\mathbb P}(T_1>t).
\end{align*}
The rightmost expression in the previous display vanishes as $t\to\infty$ for all $\epsilon>0$, as a consequence of again the Markov inequality (more specifically: ${\mathbb P}(T_1>t) \le ({{\mathbb E}D+{\mathbb E}U})/{t}$). 
\end{itemize}
Summarizing the above, we conclude
\begin{equation}
\label{NNN}\lim_{t\to\infty}\frac{\alpha(t)-\alpha^\circ t}{\sqrt{t}} \stackrel{\rm d}{=} \lim_{t\to\infty}\frac{1}{\sqrt{t}}
\sum_{k=1}^{N(t)} \Big((1-\alpha^\circ)D_i -\alpha^\circ U_i\Big).\end{equation}
We proceed by finding, asymptotically matching, upper and lower bounds on
\[p_t(x):= {\mathbb P}\left(\frac{1}{\sqrt{t}}
\sum_{k=1}^{N(t)} \Big((1-\alpha^\circ)D_i -\alpha^\circ U_i\Big)<x\right).\]
We start with an upper bound. To this end, first observe that, with $m:=1/({\mathbb E}D+{\mathbb E}U)$, for $\delta>0$,
\begin{align}\label{UPPB}
p_t(x) &\le {\mathbb P}\left(\frac{1}{\sqrt{t}}
\sum_{k=1}^{N(t)} \Big((1-\alpha^\circ)D_i -\alpha^\circ U_i\Big)<x, \frac{N(t)}{t} \ge m(1-\delta)\right) +{\mathbb P}
\left(\frac{N(t)}{t} < m(1-\delta)\right).
\end{align}
The latter probability in the right-hand side of (\ref{UPPB}) vanishes as $t\to\infty$ due to the law of large numbers, whereas the former is majorized by
\[{\mathbb P}\left(\frac{1}{\sqrt{t}}
\sum_{k=1}^{mt(1-\delta)} \Big((1-\alpha^\circ)D_i -\alpha^\circ U_i\Big) <x\right);\]
with $s^2:={\mathbb V}{\rm ar}((1-\alpha^\circ)D -\alpha^\circ U)$,  by virtue of the central limit theorem we thus find that
\[\limsup_{t\to\infty} p_t(x)\le \Phi\left(\frac{x}{\sqrt{m(1-\delta)s^2}}\right).\]
We now consider a lower bound. Note that, using ${\mathbb P}(A\cap B) \ge {\mathbb P}(A)-{\mathbb P}(B^{\rm c})$,
\begin{align*}p_t(x)\ge&\:\, {\mathbb P}\left(\frac{1}{\sqrt{t}}
\sum_{k=1}^{N(t)} \Big((1-\alpha^\circ)D_i -\alpha^\circ U_i\Big)<x, \frac{N(t)}{t} \le m(1+\delta)\right)\\
\ge&\:\,{\mathbb P}\left(\frac{1}{\sqrt{t}}
\sum_{k=1}^{mt(1+\delta)} \Big((1-\alpha^\circ)D_i -\alpha^\circ U_i\Big)<x, \frac{N(t)}{t} \le m(1+\delta)\right)\\
\ge&\:\,{\mathbb P}\left(\frac{1}{\sqrt{t}}
\sum_{k=1}^{mt(1+\delta)} \Big((1-\alpha^\circ)D_i -\alpha^\circ U_i\Big)<x\right)-{\mathbb P}\left(\frac{N(t)}{t} > m(1+\delta)\right)
\end{align*}
The latter probability vanishing as $t\to\infty$ (again due to the law of large numbers), so that we can conclude that
\[\liminf_{t\to\infty} p_t(x)\ge \Phi\left(\frac{x}{\sqrt{m(1+\delta)s^2}}\right).\]
Now letting $\delta\downarrow 0$, upon combining the above, we arrive at
\[\lim_{t\to\infty}{\mathbb P}\left(\frac{\alpha(t)-\alpha^\circ t}{\sqrt{t}}\right) =\Phi\left(\frac{x}{\sqrt{m s^2}}\right).\]
Then observe that $s^2$ equals
\begin{align*}  (1-\alpha^\circ)^2\,{\mathbb V}{\rm ar}\,D -2(1-\alpha^\circ)
\alpha^\circ {\mathbb C}{\rm ov}(D,U) +(\alpha^\circ )^2\, {\mathbb V}{\rm ar}\,U=(1-\alpha^\circ)^2\,\sigma_\alpha^2 -2(1-\alpha^\circ)
\alpha^\circ \,c +(\alpha^\circ )^2\, \sigma_\beta^2.
\end{align*}
We have thus established that
\[\lim_{t\to\infty}\frac{\alpha(t)-\alpha^\circ t}{\sqrt{v\,t}} \stackrel{\rm d}{=} Z,\:\:\:\:\mbox{
with}\:\:\:\:
v:= \frac{\beta^2\sigma_\alpha^2+\alpha^2\sigma_\beta^2-2\alpha\beta c }{(\alpha+\beta)^3}.\]
This proves the stated. 
\end{proof}

\section{Proof of Large Deviations Result}
\label{Proof of LDP}
\begin{proof}[\textbf{Proof of Theorem \ref{two equations theorem}}]
Eqn.\ (\ref{large deviation}) is an immediate consequence of the G\"artner-Ellis theorem \cite{dem}, so that we are left with 
(\ref{c(theta)}). In the proof of (\ref{c(theta)}), the following auxiliary model is used. Consider the two sequences  $(D_n)_{n\in{\mathbb N}}$ and  $(U_n)_{n\in{\mathbb N}}$ defined in Section \ref{sec:Model description}, and define ${Y}(\cdot)$ as follows. 
During the time intervals $D_i$ the process ${Y}(\cdot)$ increases with rate $(1-d)$, whereas during the time interval $U_i$ it decreases with rate $d$, for some $0<d<1$; for now $d$ is just an arbitrary positive constant larger than ${\alpha}/({\alpha+\beta})$. Note the process $Y(\cdot)$ corresponds to a fluid model with drain rate $d$ and arrival rate $1$ during periods that the driving alternating renewal process is in a $D$ period. The occupation time $\alpha(t)$ then represents the amount of input during $[0,t]$. 
Hence, $Y(t) = \alpha(t)-dt.$

Consider the decay rate 
\[ \omega = -\lim_{u\to\infty}\frac{1}{u}\Pb(\exists t>0: Y(t) > u).\]
The idea is that we find two expressions for this decay rate; as it turns out, their equivalence proves (\ref{c(theta)}). Due to similarities with existing derivations, the proofs are kept succinct.
\begin{itemize}
\item[$\circ$] Observe that, with $S_n= (1-d)D_n - dU_n$,
\[\omega = -\lim_{u\to\infty}\frac{1}{u}\log\Pb\left(\exists n\in{\mathbb N}: \sum_{i=1}^n S_i > u\right).\]
The Cram\'er-Lundberg result \cite[Section XIII.5]{Asm} implies that $\omega$ solves $\E e^{\theta(1-d)D-\theta d U}=1$.
\item[$\circ$] Denote
\[I(a)= \sup_{\theta}(\theta a -\bar\lambda(\theta)),\:\:\:\mbox{where}\:\:\:
\bar\lambda(\theta)=\lim_{t\to\infty}\frac{1}{t}\log \E e^{\theta Y(t)} = \lambda(\theta) -d\theta.\]
Observe that
\[-\omega = \lim_{u\to\infty}\frac{1}{u}\log\Pb(\exists t>0: Y(tu) > u) \ge \sup_{t>0} \left(\lim_{u\to\infty}\frac{1}{u}\log\Pb(Y(tu) > u)\right).\]
Using `G\"artner-Ellis' \cite{dem}, we thus find the lower bound
\begin{equation}
\label{omega1}-\omega \ge -\inf_{t>0} t\,I\left(\frac{1}{t}\right) = -\inf_{t>0} \frac{I(t)}{t},\end{equation}
where $\inf_{t>0}  {I(t)}/{t}$ equals \cite{Duffield, Gly}  the solution of the equation $\bar\lambda(\theta)=0$, or, alternatively, $\lambda(\theta)=d\theta$ (which we call  $\theta^*$).

Now concentrate on the upper bound, to prove that (\ref{omega1}) holds with equality. First observe that, due to the fact that the slope $Y(\cdot)$ is contained in $[-1,1]$, for any $M>0$,
\begin{eqnarray*}
\Pb(\exists t>0: Y(t) > u) &\le& \Pb(\exists n\in{\mathbb N}: Y(n) > u-1) \\
&\le& 
\sum_{n=1}^{\lceil Mu\rceil } \Pb(Y(n) > u-1) + 
\sum_{n=\lceil Mu\rceil +1}^\infty \Pb(Y(n) > u-1)\\
&\le& 
\lceil Mu\rceil \max_{n\in\{1,\ldots,\lceil Mu\rceil \}} \Pb(Y(n) > u-1) + 
\sum_{n=\lceil Mu\rceil +1}^\infty \Pb(Y(n) > u-1).\end{eqnarray*}
Note that $u^{-1}\log \lceil Mu\rceil \to 0$ and 
\[
\max_{n\in\{1,\ldots,\lceil Mu\rceil\}} 
\Pb(Y(n) > u-1) \le  \sup_{t>0} 
\Pb(Y(t) > u-1)=  \sup_{t>0} 
\Pb(Y(tu) > u-1),\]
and hence\[
\limsup_{u\to\infty}\frac{1}{u}\log \max_{n\in\{1,\ldots,\lceil Mu\rceil\}} 
\Pb(Y(n) > u-1) \le -\inf_{t>0} t\,I\left(\frac{1}{t}\right) = -\inf_{t>0} \frac{I(t)}{t}=-  {\theta^*}.\]
Also, for $u\ge 1$ and $\theta>0$, by Markov's inequality,
\[\sum_{n=\lceil Mu\rceil +1}^\infty \Pb(Y(n) > u-1)\le \sum_{n=\lceil Mu\rceil +1}^\infty \Pb(Y(n) >0)\le 
 \sum_{n=\lceil Mu\rceil +1}^\infty \E e^{\theta Y(n)}.\]
 For all $\epsilon>0$ we have for $n$ large enough $\E e^{\theta Y(n)}\le e^{(\lambda(\theta)+\epsilon)n}.$ Choose $\theta$ such that $\lambda(\theta)<-2\epsilon$ (which is possible;  $\lambda(\cdot)$ is convex and decreases in the origin), so as to obtain
 \[\limsup_{u\to\infty}\frac{1}{u}\log \sum_{n=\lceil Mu\rceil +1}^\infty \Pb(Y(n) > u-1)\le -\epsilon M,\]
 which is for sufficiently large $M$ smaller than $-\theta^*$. Combining the above, we conclude by \cite[Lemma 1.2.15]{dem} that $-\omega\le -\theta^*$, and 
hence  (\ref{omega1}) holds with equality.
\end{itemize}
From the two approaches, we see that the solutions $(\theta,d)$ of 
\[\E e^{\theta(1-d)D-\theta d U}=1\:\:\:\:\mbox{and}\:\:\:\:\lambda(\theta)=d\theta\]
coincide; it is standard to verify that for fixed $d$ both equations have a unique positive solution $\theta$ (use the convexity of moment generating functions and cumulant generating functions, whereas their slopes in $0$ are negative due to the conditions imposed on $d$).
This concludes the proof.
\end{proof}

\section{Proofs of results on spectrally positive L\'evy processes}
\begin{proof}[\textbf{Proof of Theorem \ref{Theorem on occupation time SP}}]

We first prove the result for the case of a storage model, that is, for $t\geq0$, $X(t) = A(t) - rt$, where the arrival process $A(\cdot)$ is a pure jump subordinator. As the process $X(\cdot)$ is then of bounded variation, the random variables $D_i$ and $U_i$ are strictly positive a.s., see \cite{Cohen}. The result follows as an
 application of Thm.\ \ref{theorem on LST}. Using Proposition \ref{transforms} we see that, for $\theta,q\geq0$ 
\[L_{1,2}(q+\theta,q)  =1-\frac{\theta W_{\psi(q)}^{(\theta)}(\tau)+\psi(q)Z_{\psi(q)}^{(\theta)}(\tau)}{W_{\psi(q)}^{(\theta)'}(\tau)+\psi(q)W_{\psi(q)}^{(\theta)}(\tau)} \hspace{3mm}\text{and}\hspace{3mm} L_1(q+\theta)=1-\frac{(q+\theta)W^{(q+\theta)}(\tau)}{W^{(q+\theta)'}(\tau)}.\]
Using the expressions in (\ref{exponential twist W}) we obtain 
\[\frac{1-L_1(q+\theta)}{q+\theta}+\frac{L_1(q+\theta)-L_{1,2}(q+\theta,q)}{q} = \frac{1}{q}\frac{\psi(q)Z_{\psi(q)}^{(\theta)}(\tau)}{e^{-\psi(q)\tau}W^{(q+\theta)'}(\tau)}\]
and
\[\frac{1}{1-L_{1,2}(q+\theta,q)} = \frac{e^{-\psi(q)\tau}W^{(q+\theta)'}(\tau)}{\theta W_{\psi(q)}^{(\theta)}(\tau)+\psi(q)Z_{\psi(q)}^{(\theta)}(\tau)}.\]
Multiplying the two expressions we obtain the desired result.

Now consider the case that  $X(\cdot)$ is an arbitrary spectrally positive L\'evy process. We use a limiting argument based on an approximation procedure as discussed in \cite[p. 210]{Ber} and followed in e.g.~\cite{Kyp Ref}. 
Specifically, in this case we know \cite[p. 210]{Ber} that $X(\cdot)$, with $X(t) = A(t) - rt$ for $t\geq0$, is the limit of a sequence of L\'evy processes $\{X^{(n)}(\cdot)\}_n$ with no downward jumps and bounded variation. As $n$ goes to infinity $X^{(n)}(\cdot)$ converges to $X(\cdot)$ a.s.\ uniformly on compact time intervals. Since the process $X^{(n)}(\cdot)$ has bounded variation, it can be written in the form $X^{(n)}(t) = A^{(n)}(t) - r^{(n)}t$ with $A^{(n)}(\cdot)$  a pure jump subordinator; the superscript $n$ is used when referring to quantities related to the approximating sequence $X^{(n)}(\cdot)$.  By \cite[Lemma 13.5.1]{Whi}, the reflection operator is Lipschitz continuous with respect to the supremum norm, so that the reflected processes $Q^{(n)}(\cdot)$ converge a.s.\ uniformly on compact time intervals to the process $Q(\cdot)$. First we establish the convergence, as $n\to\infty$, of the left hand side of (\ref{OTSM}), i.e.,
\[\int_0^\infty e^{-qt}\E e^{\theta \alpha^{(n)}(t)}{\rm d}t \stackrel{n\rightarrow \infty}{\rightarrow} \int_0^\infty e^{-qt}\E e^{\theta \alpha(t)}{\rm d}t.\]
This convergence stems from the continuity of the indicator function and the Dominated Convergence Theorem. Concerning the right hand side of (\ref{OTSM}), we have to prove that, for $\theta, q\geq0$ and $\tau>0$, 
\begin{equation}
\label{convergence to prove}
\frac{\psi^{(n)}(q)Z_{\psi^{(n)}(q)}^{(\theta,n)}(\tau)}{\theta W_{\psi^{(n)}(q)}^{(\theta,n)}(\tau)+\psi^{(n)}(q)Z_{\psi^{(n)}(q)}^{(\theta,n)}(\tau)}\stackrel{n\rightarrow \infty}{\rightarrow}\frac{\psi(q)Z_{\psi(q)}^{(\theta)}(\tau)}{\theta W_{\psi(q)}^{(\theta)}(\tau)+\psi(q)Z_{\psi(q)}^{(\theta)}(\tau)}.
\end{equation}
Since $X^{(n)}(\cdot)\to X(\cdot)$ a.s.\ uniformly on compact intervals,  also the Laplace exponent $\phi^{(n)}(\cdot)$ converges pointwise to the Laplace exponent $\phi(\cdot)$ and thus, $\psi^{(n)}(\cdot)$ will also converge pointwise to $\psi(\cdot)$. Due to the Continuity Theorem for Laplace transforms, or using direclty \cite[Eqn.\ (3.2)]{Kyp Ref},  
\[W^{(\theta,n)}(x) \stackrel{n\rightarrow \infty}{\rightarrow} W^{(\theta)}(x) \hspace{5mm} \text{ for all} \hspace{2mm} x\geq0.\]
This yields, for all $x\geq0$, the convergence $W_{\psi^{(n)}(q)}^{(\theta,n)}(x) \rightarrow  W_{\psi(q)}^{(\theta)}(x)$ as $n\to\infty$.
Using the Bounded Convergence Theorem and \cite[Lemma 3.3]{Hub},
\[Z_{\psi^{(n)}(q)}^{(\theta,n)}(x) \stackrel{n\rightarrow \infty}{\rightarrow} Z_{\psi(q)}^{(\theta)}(x) \hspace{5mm} \text{ for all} \hspace{2mm} x\geq0.\]
These results prove the convergence in (\ref{convergence to prove}).
\end{proof}

\begin{proof}[\textbf{Proof of Corollary \ref{distributional equality}}]
From the Wiener-Hopf factorization we have that, for $\theta>0$ and an exponentially distributed random variables $e_q$,
\[\E e^{-\theta(e_q-G_{e_q})} = \frac{\bar{k}(q,0)}{\bar{k}(q+\theta,0)}=\frac{\psi(q)}{\psi(q+\theta)},\:\:\:\:\mbox{where }\:\:\:
\frac{\bar{k}(q,0)}{\bar{k}(q+\theta,0)} : = \E\left(e^{-\theta (e_q - G_{e_q})}\right);\]
the second equation  is due to explicit expression of the Wiener-Hopf factors for spectrally one-sided L\'evy processes, as can be found in e.g.\ \cite{Man,Kyp}. In what follows we prove that, for $\theta,q\geq0$,
\begin{equation}
\label{Tauber}
\lim_{\tau\rightarrow\infty} \frac{\psi(q)Z_{\psi(q)}^{(\theta)}(\tau)}{\theta W_{\psi(q)}^{(\theta)}(\tau)+\psi(q)Z_{\psi(q)}^{(\theta)}(\tau)}=\frac{\psi(q)}{\psi(q+\theta)}.
\end{equation}
In order to show (\ref{Tauber}) we extend \cite[Lemma 3.3]{Kuz} to the scale functions $W_{\psi(q)}^{(\theta)}(x)$ and $Z_{\psi(q)}^{(\theta)}(x)$. We have
\[W_{\psi(q)}^{(\theta)}(x) = e^{-\psi(q)x}W^{(q+\theta)}(x),\]
and by \cite[Lemma 3.3]{Kuz}  
\[\lim_{x\rightarrow\infty} e^{-\psi(q+\theta) x}W^{(q+\theta)}(x) = \frac{1}{\phi'(\psi(q+\theta))},\]
and hence
\begin{equation}
\label{Tauber2}
\lim_{x\rightarrow\infty} e^{-(\psi(q+\theta)-\psi(q))x}W_{\psi(q)}^{(\theta)}(x) = \frac{1}{\phi'(\psi(q+\theta))}.
\end{equation}
Next we study the  behavior of the scale function $Z_{\psi(q)}^{(\theta)}(x)$ for  $x$ large. Define the function $U(\cdot)$ by
\[U(x) = e^{-(\psi(q+\theta)-\psi(q))x}\left(Z_{\psi(q)}^{(\theta)}(x)-1\right).\] Using the definition of the $q$-scale function, as in \cite{Ber2} and \cite[Eqn.\ (3.5)]{Kuz}, it is straightforward to show that $U(\cdot)$ is an increasing function. Then the measure $U({\rm d}x)$ has Laplace transform, for $\lambda\geq0$,
\[\mathscr{L}U(\lambda)=\int_0^\infty e^{-\lambda x} U({\rm d}x) = \frac{\lambda\theta}{(\lambda+\psi(q+\theta)-\psi(q))(\phi(\lambda+\psi(q+\theta))-(q+\theta))}.\]
When letting $\lambda\downarrow0$ we find 
\[\lim_{\lambda\downarrow0}\mathscr{L}U(\lambda) = \frac{1}{\phi'(\psi(q+\theta))}\frac{\theta}{(\psi(q+\theta)-\psi(q))}.\]
Applying Karamata's Tauberian Theorem \cite[Thm.\ 1.7.1]{Bing},  
\begin{equation}
\label{Karamata}
 \lim_{x\rightarrow\infty}e^{-(\psi(q+\theta)-\psi(q))x}Z_{\psi(q)}^{(\theta)}(x) =\frac{1}{\phi'(\psi(q+\theta))}\frac{\theta}{(\psi(q+\theta)-\psi(q))} .
\end{equation}
Upon combining (\ref{Tauber2}) and (\ref{Karamata}),
\[\lim_{x\rightarrow\infty} \frac{\theta W_{\psi(q)}^{(\theta)}(x)}{\psi(q)Z_{\psi(q)}^{(\theta)}(x)} = \frac{\psi(q+\theta)-\psi(q)}{\psi(q)}\]
which after some straightforward algebra leads to (\ref{Tauber}). By uniqueness of the Laplace transform we find that $\alpha(e_q)$ and $e_q  - G_{e_q}$ are 
equal in distribution.
\end{proof}

\section*{Acknowledgements}

The research of N. Starreveld and M. Mandjes is partly funded by the NWO Gravitation project
N{\sc etworks}, grant number 024.002.003.

\appendix
\section{Appendix}
\label{sec: Additional results}

\begin{proof}[\textbf{Proof of Proposition \ref{corollary AF}}]
The proof is along the same lines as the proof in Section \ref{proof of LST}; we only need to condition on $X(\cdot)$ after an exponentially distributed amount of time. Given $X(0)\in A$, we have 
\begin{align*}
\E \left[e^{-\theta\alpha(e_q)}1_{\{X(e_q)\in A\}}\right] &= \E \left[e^{-\theta \alpha(e_q)}1_{\{e_q<D\}}\right]+\E\left[e^{-\theta\alpha(e_q)}1_{\{e_q\geq D+U\}}1_{\{X(e_q)\in A\}}\right]\\
&= \frac{q}{q+\theta}\left(1-\textcolor{black}{L_1(q+\theta,q)}\right)+ L_{1,2}(q+\theta,q)\cdot\E \left[e^{-\theta\alpha(e_q)}1_{\{X(e_q)\in A\}}\right]
\end{align*}
Hence,
\[\E \left[e^{-\theta\alpha(e_q)}1_{\{X(e_q)\in A\}}\right] =\frac{q}{q+\theta}\frac{1-L_1(q+\theta)}{1-L_{1,2}(q+\theta,q)}.\]
Substituting $\theta=0$,
\[\int_{0}^{\infty}e^{-qt}\Pb(X(t)\in A|X(0)\in A){\rm d}t =\frac{1}{q}\frac{1-L_1(q)}{1-L_{1,2}(q,q)}.\]
Similarly, to obtain (\ref{transform AF2}) we have, given $X(0)\in A$,
\begin{align*}
\E \left[e^{-\theta\alpha(e_q)}1_{\{X(e_q)\in B\}}\right] &= \E \left[e^{-\theta \alpha(e_q)}1_{\{D\leq e_q<D+U\}}\right]+\E\left[e^{-\theta\alpha(e_q)}1_{\{e_q\geq D+U\}}1_{\{X(e_q)\in B\}}\right]\\
&=L_1(q+\theta)-L_{1,2}(q+\theta,q) + L_{1,2}(q+\theta,q)\cdot\E \left[e^{-\theta\alpha(e_q)}1_{\{X(e_q)\in B\}}\right],
\end{align*}
yielding
\[\E \left[e^{-\theta\alpha(e_q)}1_{\{X(e_q)\in B\}}\right] =\frac{L_1(q+\theta)-L_{1,2}(q+\theta,q)}{1-L_{1,2}(q+\theta,q)}.\]
Substituting $\theta=0$, we obtain (\ref{transform AF2}).
\end{proof}

\begin{lemma}
\label{lemma conv W}
The $q$-scale function $W^{(q)}(x)$ satisfies 
$\partial_q W^{(q)}(x) = (W^{(q)}\star W^{(q)})(x).$
\end{lemma} 
\begin{proof}[\textbf{Proof of Lemma \ref{lemma conv W}}]
By definition of the $q$-scale function and the results of \cite[Section 3.3]{Kuz}, we have that 
\[\int_0^\infty e^{-\theta x}W^{(q)}(x){\rm d}x =\frac{1}{\phi(\theta)-q}=\int_0^\infty \sum_{k=0}^\infty e^{-\theta x}  q^k W^{\star (k+1)}(x) \,{\rm d}x.\] 
Therefore,
\[\lim_{h\downarrow 0} \int_0^\infty e^{-\theta x} \Big(\frac{W^{(q+h)}(x)-W^{(q)})(x)}{h}\Big) \,{\rm d}x=\lim_{h\downarrow 0} \int_0^\infty e^{-\theta x} \sum_{k=0}^\infty \frac{(q+h)^k -q^k}{h}W^{\star (k+1)}(x) \,{\rm d}x,\]
where $W^{\star k}(\cdot)$ is the $k$-fold convolution of $W^{(0)}(\cdot).$ Because of the convexity of $x\mapsto x^k$, for any $k\ge 1$,
\[k\,q^{k-1}\le \frac{(q+h)^k - q^k}{h} \le k(q+h)^{k-1} .\]
Again using  \cite[Section 3.3]{Kuz}, it is a matter of calculus to verify that, 
\[\int_0^\infty e^{-\theta x} \sum_{k=0}^\infty k\,q^{k-1}W^{\star (k+1)}(x) \,{\rm d}x=\frac{1}{(\phi(\theta)-q)^2},\]
whereas for $h\in (0, \phi(\theta)-q)$,
\[\int_0^\infty e^{-\theta x} \sum_{k=0}^\infty k(q+h)^{k-1}W^{\star (k+1)}(x) \,{\rm d}x=\frac{1}{(\phi(\theta)-q-h)^2},\]
which converges to $(\phi(\theta)-q)^{-2}$ as $h\downarrow 0.$

Computing the Laplace transform of the convolution $(W^{(q)}\star W^{(q)})(x)$, we also find 
\[\int_0^\infty e^{-\theta x}(W^{(q)}\star W^{(q)})(x){\rm d}x =  \frac{1}{(\phi(\theta)-q)^2}.\]
Combining the above, the result follows by the uniqueness of the Laplace transform. 
\end{proof}

\begin{proof}[\textbf{Proof of Proposition \ref{transforms}}]

We first prove  (\ref{transformUD}). We start by showing that 
\begin{equation}
\label{transform equation}
\E e^{-\theta_1 D -\theta_2 U} = \E e^{-\theta_1\tau_\alpha -\psi(\theta_2)(Q(\tau_{\alpha})-\tau)}.
\end{equation}
Due to (\ref{D time}) and (\ref{U time}),
$\E \text{exp}(-\theta_1 D -\theta_2 U)= \E \text{exp} (-\theta_1\tau_\alpha-\theta_2\sigma_{Q(\tau_\alpha)}^{(+)}).$
Conditioning on the value of the first exit time from $[0,\tau]$, i.e., $\tau_\alpha$, and the overshoot over level $\tau$ at that time, i.e. $Q(\tau_\alpha)$, we find
\begin{align*}
\E e^{-\theta_1\tau_\alpha-\theta_2\sigma_{Q(\tau_\alpha)}^{(+)}}&= \int_{(0,\infty)}\int_{(\tau,\infty)} \E e^{-\theta_1 t-\theta_2 \sigma_{s}^{(+)}}\Pb\left(Q(t)\in{\rm d}s,\tau_\alpha\in{\rm d}t\right)\\
&= \int_{(0,\infty)} e^{-\theta_1 t}\int_{(\tau,\infty)} \E e^{-\theta_2 \sigma_{s}^{(+)}}\Pb\left(Q(t)\in{\rm d}s,\tau_\alpha \in{\rm d}t\right) \numberthis \label{exceed above}.
\end{align*}
Now we use the transform of the first-exit time $\sigma_{s}^{(+)}$ established in \cite[Eqn.\ (3)]{Pal}; note that in 
\cite{Pal} this result has been derived for a spectrally negative L\'evy process but a similar argument yields the result for the spectrally positive case. For a spectrally positive L\'evy process with Laplace exponent $\phi(\theta)$ we consider the exponential martingale $\mathcal{E}_t(c) = e^{-cX_t-\phi(c)t}$ and then the result follows by invoking the same arguments as in the spectrally negative case. Omitting a series of mechanical steps, this eventually yields
\begin{equation}
\label{transform above}
\E e^{-\theta_2\sigma_{s}^{(+)}} = e^{-\psi(\theta_2)(s-\tau)}.
\end{equation}
Substituting (\ref{transform above}) into (\ref{exceed above}) we obtain the right hand side of (\ref{transform equation}), as desired. 

For the right hand side of (\ref{transformUD}) we use \cite[Thm.\ 1]{Avr}, which provides the joint transforms of the first-exit time and exit position from $[0,\tau]$. The expressions in (\ref{mean time}) are obtained in the usual way: differentiating (\ref{transformUD}) and inserting 0. An expression for $\E[DU]$ is found similarly; during these computations we need Lemma \ref{lemma conv W}. A similar reasoning applies for (\ref{variance D}), (\ref{variance U}), and  (\ref{Covariance}). \end{proof}

\end{document}